\documentclass[12pt]{amsart}
\usepackage{amsfonts,latexsym,rawfonts,amsmath,amssymb,amsthm}
\usepackage[plainpages=false]{hyperref}
\usepackage{esint}
\usepackage{tikz}
\usetikzlibrary{calc, arrows.meta}
\usepackage{float}

\usepackage{tikz}
\usepackage{mathdots}
\usepackage{yhmath}
\usepackage{cancel}
\usepackage{siunitx}
\usepackage{multirow}
\usepackage{tabularx}
\usepackage{extarrows}
\usepackage{booktabs}
\usetikzlibrary{fadings}
\usetikzlibrary{patterns}
\usetikzlibrary{shadows.blur}
\usetikzlibrary{shapes}

\usepackage{graphicx}
\RequirePackage{color}

\numberwithin{equation}{section}

\newtheorem{prop}{Proposition}[section]
\newtheorem{thm}[prop]{Theorem}
\newtheorem{lem}[prop]{Lemma}

\newtheorem{cor}[prop]{Corollary}
\newtheorem{rem}[prop]{Remark}

\newtheorem{defi}[prop]{Definition}

\renewcommand{\div}{\mbox{div}\,}
\renewcommand{\det}{\operatorname{det}}

\allowdisplaybreaks
\title[Local potential and H\"older estimates for LMA equations]{Local potential and H\"older estimates for the linearized Monge-Amp\`ere equation}

\author{Guoqing Cui}
\author{Ling Wang}
\author{Bin Zhou}

\address{School of Mathematical Sciences, Peking University, Beijing 100871, China}
\email{gqcui25@stu.pku.edu.cn}
\address{Department of Decision Sciences and BIDSA, Bocconi University, Milano, Italy}
\email{ling.wang@unibocconi.it}
\address{School of Mathematical Sciences, Peking University, Beijing 100871, China}
\email{bzhou@pku.edu.cn}

\begin{document}



\begin{abstract}
   In this paper, we establish local potential estimates and H\"older estimates for solutions of linearized Monge-Amp\`ere equations with the right-hand side being a signed measure, under suitable assumptions on the data. In particular, the interior H\"older estimate holds for an inhomogeneous linearized Monge-Amp\`ere equation with right-hand side being the nonnegative divergence of a bounded vector field in all dimensions. As an application, we give a new approach for the interior estimate of the singular Abreu equation.
\end{abstract}

\maketitle

\section{Introduction}
In this paper, we investigate local properties of weak solutions to the linearized Monge-Amp\`ere equations with signed measure data
\begin{equation}
   -\sum_{i,j=1}^{n}U^{ij}D_{ij}v=\mu \label{eq: div LMA} 
\end{equation}
in a bounded convex domain $\Omega\subset\mathbb{R}^n$ $(n\ge 2)$, where $U=(U^{ij})$ is the cofactor matrix of the Hessian matrix of a convex function $u\in C^2(\Omega)$, satisfying 
\begin{equation}
    0<\lambda\le \operatorname{det}D^2u\le \Lambda\quad\text{in }\Omega.\label{condition}
\end{equation} 
 
In a celebrated work, Caffarelli and Gutiérrez established the Harnack inequality and Hölder continuity for solutions to the homogeneous linearized Monge-Ampère equation under the $\mathcal{A}_{\infty}$ condition \cite{CG}. Notably, the $\mathcal{A}_{\infty}$ condition is satisfied when \eqref{condition} holds. For the inhomogeneous case in \eqref{eq: div LMA} with $\mu = f$, an integrable function, Hölder estimates, higher-order regularities, and boundary behaviors have been studied under suitable assumptions on $f$; see \cite{GN1, GN2, LN1, LN2, LS, TW2}. When the right-hand side has more singularity, the equations appear in many circumstances, including the dual semigeostrophic equations \cite{ACDF,ACDF14,BB}, the polar factorization \cite{Lo}  and the singular Abreu equations \cite{Le2}, etc.
One of the main purposes in this paper is to use a potential theoretic approach to investigate the local behavior of solutions to \eqref{eq: div LMA} with $\mu$ being a signed Radon measure. 

In the potential theory of linear and nonlinear equations, numerous important results have been established over the past several decades. For instance, for the $p$-Laplacian equation, the following two-sided pointwise Wolff potential estimate is well known: if $u$ is a nonnegative superharmonic function satisfying
$$
-\operatorname{div} (|Du|^{p-2}Du) = \mu \geq 0 \quad \text{in } B(x_0, 4r),
$$
then
$$
C_{1}\mathbf{W}_{1,p}^{\mu}(x_{0},r) + \inf_{B(x_0,2r)}u \leq u(x_{0}) \leq C_{2}\inf_{B(x_0,r)}u + C_{3}\mathbf{W}_{1,p}^{\mu}(x_{0},2r),
$$
where $C_1$, $C_2$, and $C_3$ are positive constants depending on $n$ and $p$, and $\mathbf{W}_{1,p}^{\mu}(x_{0},r)$ denotes the Wolff potential of $\mu$.
This two-sided estimate was first established by Kilpeläinen and Malý \cite{KM1, KM}, using a carefully constructed test function and a clever iteration technique. Later, Trudinger and Wang \cite{TW1} provided an alternative proof using Poisson modification techniques with Harnack’s inequality. The Wolff potential estimate for solutions to degenerate equations with signed measure data was derived by  \cite{DM} and \cite{TW09}. More recently, Hara \cite{Ha} presented a new proof of such estimates for solutions to degenerate and singular elliptic equations with signed measure data.

Since $U$ is divergence free, i.e. $\displaystyle\sum_{j=1}^nD_j(U^{ij})=0$ for all $i=1,2,\cdots ,n$, we may rewrite  \eqref{eq: div LMA} in the divergence form as follows:
\begin{equation}\label{eq: +- form}
     -D_j(U^{ij}D_iv)=\mu.
\end{equation}
Here and throughout, we adopt the Einstein summation convention over repeated indices. Inspired by Hara’s approach, we establish a potential estimate for the linearized Monge-Ampère equation with signed measure data in this paper. 
Given a signed Radon measure $\mu$, we apply the Hahn-Jordan decomposition to write 
\begin{equation*}
    \mu:=\mu_+ - \mu_- \quad \text{in } \Omega, \end{equation*}
where $\mu_+$ and $\mu_-$ are nonnegative Radon measures belonging to the dual space $(W_0^{1,2}(\Omega))^{\ast}$. We denote the total variation of $\mu$ by $$|\mu| := \mu_+ + \mu_-.$$
For a measurable set $A \subset \Omega$, we denote the average of a function $f$ over $A$ by
\[
\fint_{A} f\,\mathrm{d}x := \frac{1}{|A|} \int_{A} f\,\mathrm{d}x.
\]
We also define the \emph{section} of $u$ centered at $x \in \Omega$ with height $h > 0$ as
\[
S_u(x,h) := \left\{ y \in \Omega \,:\, u(y) < u(x) + Du(x) \cdot (y - x) + h \right\}.
\]
It is well-known that under assumption \eqref{condition}, sections are equivalent to Euclidean balls in terms of geometry and measure-theoretic properties.
Then we have the following potential estimate:

\begin{thm}[Local potential estimate]\label{thm:potential est}
    Let $u\in C^2(\Omega)$ be a convex function satisfying \eqref{condition}, $v$ be a weak solution of \eqref{eq: +- form}, $x_0$ be a Lebesgue point of $v$ and $S_u(x_0,2h_0)\Subset\Omega$. Then for any $p>0$, there exists $C>0$ depending only on $n,\lambda,\Lambda$ and $p$ such that \begin{equation}
        v_{\pm}(x_0)\le C\left(\fint_{S_u(x_0,h_0)\setminus\overline{S_u(x_0,h_0/2)}} v_{\pm}^{p} \,\mathrm{d}x\right)^{1/p}+CI_u^{\mu_{\pm}}(x_0,2h_0),\label{eq:potential est}
    \end{equation}where $v_{+}=\max\{v,0\}$, $v_{-}=\max\{-v,0\}$ and $I_u^{\mu_\pm}$ is the Riesz potential with respect to  $\mu_{\pm}$ (see \eqref{eq:Riesz potential}).
    \end{thm}

\vskip 7pt

The potential estimate \eqref{eq:potential est} can be used to establish Hölder continuity for solutions of the non-homogeneous linearized Monge-Ampère equation. In particular, this estimate yields an $L^p$-$L^\infty$ bound, which plays a crucial role analogous to its counterpart in the De Giorgi-Nash-Moser theory for uniformly elliptic equations \cite{De,Na,Mo}. However, the classical De Giorgi-Nash-Moser approach does not apply directly in our setting due to the non-uniform ellipticity of the operator (see \cite[Remark 3.4]{TiW}). This is the main reason we employ the potential estimate. 



\begin{thm}[Interior H\"older estimate]\label{thm: Holder for signed data}
	Let $u\in C^2(\Omega)$ be a convex function satisfying \eqref{condition}. Suppose that there exist $M>0$ and $\varepsilon>0$ such that 
	\begin{equation}\label{eq: mu growth}
		|\mu|(S_u(x,h))\le M h^{\frac{n}{2}-1+\varepsilon}
	\end{equation} whenever $x\in \Omega$ such that $S_u(x,h)\subset\Omega$. Given a section $S_u(x_0,2h_0)\Subset\Omega$. Let $v$ be a solution to \eqref{eq: div LMA}
	in $S_u(x_0,2h_0)$ and $p\in(0,+\infty)$. Then there exist $\gamma\in(0, 1)$ depending only on $n$, $\varepsilon$, $\lambda$ and $\Lambda$, and  $C>0$ depending only on $n$, $p$, $\lambda$, $\Lambda$, $h_0$ and $\operatorname{diam}(\Omega)$, such that
	\begin{equation}\label{eq: holder under total variation}
		|v(x)-v(y)|\leq C\left(\|v\|_{L^p(S_u(x_0,2h_0))}+M\right)|x-y|^\gamma,\ \ 
		\forall x,y\in S_u(x_0,h_0).
	\end{equation}
\end{thm}

When $\mu$ is nonnegative, Theorem \ref{thm: Holder for signed data} follows directly by combining the potential estimate with the Harnack inequality of Caffarelli-Gutiérrez for the homogeneous equation (see Theorem \ref{thm:Holder by potential est for +}). 
To establish the result for a general signed measure $\mu$, we adopt a different approach based on Campanato space techniques and a weak $(1,2)$-Poincar\'e inequality adapted to the linearized Monge-Amp\`ere equation. These tools allow us to derive the desired $L^{\infty}$-$C^{\alpha}$ estimate.

\begin{rem}
    When the right-hand side $\mu=\mu_f:=f\in L^{q}(\Omega)$ with $q>\frac{n}{2}$, we can verify that  
    \begin{align}\label{eq: f-growth}
        |\mu_f|(S_u(x,h))&=\int_{S_u(x,h)}|f|\,\mathrm{d}x\leq \|f\|_{L^{q}(\Omega)}|S_u(x,h)|^{1-\frac{1}{q}} \leq C\|f\|_{L^{q}(\Omega)}h^{\frac{n}{2}-1+\varepsilon},
    \end{align}
    where $\varepsilon=1-\frac{n}{2q}>0$. Consequently, we recover the interior H\"older regularity of the solution $v$ to linearized Monge-Amp\`ere equation with $L^{\frac{n}{2}+}$-inhomogeneity, which was first proved by Le-Ngyuen using Green's function \cite{LN2}.
\end{rem}

As a direct application of Theorem~\ref{thm: Holder for signed data}, we may obtain a Hölder estimate for solutions to the linearized Monge-Ampère equation when the right-hand side measure $\mu$ is given by $\operatorname{div} \mathbf{F} + f$, where $\mathbf{F} := (F^1(x), \dots, F^n(x)) : \Omega \to \mathbb{R}^n$ is a vector field with $\operatorname{div }\mathbf{F}\ge 0$ and $f: \Omega\to\mathbb R$ is a function. The precise statement is as follows.

\begin{thm}[Interior H\"older estimate with right-hand side in divergence form]\label{thm:divF+f}
    Let $u\in C^2(\Omega)$ be a convex function satisfying \eqref{condition}. Let $v$  be a solution to
    \begin{equation}
   -U^{ij}D_{ij}v=\operatorname{div}\mathbf{F}+f, \label{eq: div LMA-Ff} 
\end{equation}
     where $\mathbf{F}\in L^{\infty}(\Omega;\mathbb{R}^n)$ with $\operatorname{div }\mathbf{F}\ge 0$ and $f\in L^{q}(\Omega)$ with $q>n/2$. Let $p\in (0,+\infty)$ and $S_u(x_0,2h_0)\Subset\Omega$. Then there exist $\gamma>0$ depending only on $n$, $\lambda$, $\Lambda$, $p$, and $q$, and constant $C>0$ depending only on $n$, $p$, $q$, $\lambda$, $\Lambda$, $h_0$ and $\operatorname{diam}(\Omega)$, such that 
        \begin{equation}\label{eq: holder div F f}
        |v(x)-v(x_0)|\leq C\left(\|v\|_{L^p(S_u(x_0,2h_0))}+\|\mathbf{F}\|_{L^{\infty}(\Omega)}+\|f\|_{L^{q}(\Omega)}\right)|x-x_0|^\gamma,
    \end{equation}
     for all $x\in S_u(x_0,h_0)$.
\end{thm}

\begin{rem}
This corollary is analogous to the De Giorgi-Nash-Moser theory for general divergence-form uniformly elliptic equations. The estimate was previously obtained in \cite{Lo} under the assumption that $\det D^2 u$ is sufficiently close to a constant, and in \cite{Le1} for the two-dimensional case. The H\"older regularity for \eqref{eq: div LMA-Ff} under the structural condition \eqref{condition} in higher dimensions has also been established assuming additional integrability conditions on $D^2u$ or $(D^2u)^{\frac{1}{2}}\mathbf{F}$; see \cite{Ki, W}. Our result on H\"older regularity for \eqref{eq: div LMA-Ff} in Theorem~\ref{thm:divF+f} needs the non-negativity of the $\operatorname{div }\mathbf{F}$ in all dimensions. So far, we haven't come up with a way to deal with general $\mathbf{F}$.
\end{rem}

Finally, we remark that equations of the form \eqref{eq: div LMA-Ff} arise from singular Abreu equations \cite{KLWZ, Le2, LZ}, which appear in the study of convex functionals with a convexity constraint related to the Rochet–Chon'e model for the monopolist problem in economics. These equations provide an important analytical framework for understanding variational problems involving convex potentials and have recently attracted considerable attention.
In particular, in \cite{KLWZ}, Kim, Le, together with the second and third authors, established the regularity of a class of singular Abreu equations by transforming the original fourth-order equation into a linearized Monge–Amp\`ere equation with a drift term. This transformation makes it possible to apply the techniques developed for degenerate elliptic equations to the study of these highly nonlinear problems.
Theorem~\ref{thm:divF+f} offers a new and more direct approach to obtaining interior regularity results for singular Abreu equations. It provides an alternative viewpoint that complements the method in \cite{KLWZ}, and the main idea of this approach will be briefly outlined in the final section.

Another source of equations of the form \eqref{eq: div LMA-Ff} is the study of semigeostrophic equations \cite{ACDF, Le1, Lo}.   In particular, the  interior H\"older estimates for the time derivatives of solutions to the dual semigeostrophic equations in dimension two with the initial potential density is bounded away from zero and infinity was settled by Le \cite{Le1}, through the study on \eqref{eq: div LMA-Ff}. The three dimensional case of this problem is interesting and remains open. 
Unfortunately, Theorem \ref{thm:divF+f} cannot be directly applied due to the restriction on the sign of $\operatorname{div }\mathbf{F}$.

The rest of the paper is organized as follows. In Section~\ref{sec:pre}, we collect some useful definitions and lemmas that will be used later. Section~\ref{sec:potential} is devoted to proving the potential estimate \eqref{eq:potential est} stated in Theorem~\ref{thm:potential est}. The Hölder regularity result and the proof of Theorem~\ref{thm:divF+f} are presented in Section~\ref{sec:Holder1}. Finally, in the last section, we apply Theorem~\ref{thm:divF+f} to the regularity of the singular Abreu equation.

\vskip 7pt

\section{Preliminaries}\label{sec:pre}

\vskip 7pt

In this section, we collect some fundamental results that will be used in the subsequent sections. Some of these results can be found in the existing literature or are straightforward to prove. Throughout this section, we assume that $\Omega$ is a bounded domain in $\mathbb{R}^n$ with smooth boundary, and that $u \in C^2(\Omega)$ is a convex function satisfying \eqref{condition}.

\subsection{Lorentz spaces} Firstly, we recall the definition of Lorentz spaces.
\begin{defi}
    For any $0<p,q\leq \infty$, we define 
        $$L^{p,q}(\Omega):=\{f:\Omega\to\mathbb{R} \text{ measurable}: \Vert f\Vert_{L^{p,q}(\Omega)}<\infty\}.$$
where
    \[\Vert f\Vert_{L^{p, q}(\Omega)}:=\left\{\begin{aligned}
    &p^{\frac{1}{q}}\left(\displaystyle\int_0^{\infty}t^q\left|\{x\in\Omega: |f(x)|\ge t\}\right|^{q/p}\frac{\mathrm{d}t}{t}\right)^{\frac{1}{q}}, && q<\infty, \\[7pt]
    &\sup_{t>0}\left\{ t\left|\{x\in\Omega: |f(x)|\ge t\}\right|^{1/p}\right\},&& q=\infty.
    \end{aligned}\right.\]
\end{defi}

\vskip 10pt
We remark that when $p<\infty$, $q=\infty$, the Lorentz space is the weak-$L^p$ space, and $L^{p,p}(\Omega)=L^{p}(\Omega)$ for $0<p\leq \infty$.

The following are some useful properties that will be used in the subsequent sections.
\begin{enumerate}
\item [(i)] If $0\le f\le g$ in $\Omega$, we have 
\begin{equation}
    \Vert f\Vert_{L^{p, q}(\Omega)}\le \Vert g\Vert_{L^{p,q}(\Omega)}.\label{property 1}
\end{equation} 

\item [(ii)]  (quasi-triangle inequality) For $0<p,q<\infty$, $f, g\in L^{p,q}(\Omega)$, $f+g\in L^{p,q}(\Omega)$, and there exists $C_{p,q}=2^{1/p}\max\{1,2^{1/q-1}\}>0$ such that
\begin{equation}
     \Vert f+g\Vert_{L^{p,q}(\Omega)}\le C_{p,q} ( \Vert f\Vert_{L^{p,q}(\Omega)}+  \Vert g\Vert_{L^{p,q}(\Omega)}).\label{property 2}
\end{equation} 
    
   \item [(iii)] (H\"older-type inequality) For $1\leq p,  p_1, p_2<\infty$ and $1<q, q_1, q_2\leq\infty$ satisfying 
    \[\frac{1}{p_1}+\frac{1}{p_2}=\frac{1}{p}, \ \ \frac{1}{q_1}+\frac{1}{q_2}=\frac{1}{q},\] there exists $C_{p_1, p_2, q_1, q_2}>0$, such that 
    \begin{equation}
    \Vert fg\Vert_{L^{p,q}(\Omega)}\le C_{p_1, p_2, q_1, q_2}\Vert f\Vert_{L^{p_1,q_1}(\Omega)}\cdot\Vert g\Vert_{L^{p_2,q_2}(\Omega)}.\label{property 3: Holder ineq}
    \end{equation}
\end{enumerate}
    
\subsection{Monge-Amp\`ere Sobolev inequality}
Consider the $L^2$-norm $\|Dv\|_u$ defined by
$$
\|Dv\|_u := \left( \int_{\Omega} U^{ij} D_iv D_jv\,\mathrm{d}x \right)^{1/2},
$$
where $(U^{ij})$ is the cofactor matrix of the Hessian $D^2u$ for a convex potential $u$. This norm includes the classical case as a special example: when $u(x) = \frac{1}{2}|x|^2$, we have $D^2u = I_n$ and hence $U^{ij} = I_n$. The following Monge-Amp\`ere Sobolev inequality was established by Tian and Wang \cite{TiW} for $n \ge 3$ and by Le \cite{Le1} for $n = 2$ (see also \cite{Ma1} for some extensions and \cite{WZ24} for a complex version).
\begin{lem}[Monge-Amp\`ere Sobolev inequality]\label{lem:Sobolev inequality}
Let $u\in C^2(\Omega)$ be a convex function satisfying \eqref{condition}.
    Let $$p=\frac{2n}{n-2}\quad \text{if } n\ge 3\quad \text{and}\quad p\in (2,\infty)\quad \text{if }n=2.$$ There exists $C>0$ depending only on $n$, $p$, $\lambda$, $\Lambda$ such that for any $v\in W_0^{1,2}(\Omega)$, 
    \begin{equation*}
        \Vert v\Vert_{L^{p}(\Omega)}\le C \Vert Dv\Vert_{u}.
    \end{equation*}
\end{lem}

With the fact $L^p(\Omega)\subset L^{p,\infty}(\Omega)$, we immediately have
 \begin{equation}
        \Vert v\Vert_{L^{p,\infty}(\Omega)}\le C \Vert Dv\Vert_{u}.\label{eq:Sobolev}
    \end{equation}

We will also need the following weak (1,2)-type Poincar\'e inequality proved by \cite[Theorem 1.3]{Ma2}. For $x_0\in \Omega$ and $h>0$ with $S_u(x, h)\subset\Omega$, we denote
\[v_{x_0,h}:=\displaystyle\fint_{S_u(x_0,h)}v\,\mathrm{d}x.\]
\begin{lem}[Monge-Amp\`ere Poincar\'e inequality]\label{lem:Poincare}
	Let $\Omega$ be a bounded domain in $\mathbb{R}^n$ with smooth boundary and $u\in C^2(\Omega)$ be a convex function satisfying \eqref{condition}. There exists $C>0$ depending only on $n$, $\lambda$ and $\Lambda$ such that for any $v\in W^{1,2}(\Omega)$ and $S_u(x_0,h)\subset\Omega$, 
	\begin{eqnarray*}
		\fint_{S_u(x_0,h)}\left|v-v_{x_0,h}\right|\,\mathrm{d}x\le Ch^{1/2}\bigg(\fint_{S_u(x_0,h)}U^{ij}D_ivD_jv\,\mathrm{d}x\bigg)^{1/2}.
	\end{eqnarray*}
\end{lem}

In fact, according to \cite{Ma2}, the inequality holds when the Monge-Amp\`ere measure $\mu_u$ satisfies the doubling property \eqref{eq: doubling property}.

\subsection{The homogeneous linearized Monge-Amp\`ere equation}
Combining the crucial decay estimate of Caffarelli and Gutiérrez \cite{CG} with an argument analogous to that used in Theorem 4.8 of Caffarelli and Cabré \cite{CC}, we  can obtain local boundedness and the weak Harnack inequality for linearized Monge-Ampère equations.
\begin{lem}[Local boundedness]\label{lem:local boundedness}
    Let $u\in C^2(\Omega)$ be a convex function satisfying \eqref{condition} and $v\in W^{1,2}(\Omega)$ be a weak nonnegative subsolution to 
    \[-U^{ij}D_{ij}v=0\]
     in $\Omega$ with \eqref{condition}. Then for any section $S_u(x_0,h)\Subset \Omega$ there holds for any $p>0$ , $$\sup_{S_u(x_0,h/2)} v\le C\left(\fint_{S_u(x_0,h)}v^{p}\,\mathrm{d}x\right)^{1/p},$$ where $C>0$ depends only on $n$, $\lambda$, $\Lambda$ and $p$.
\end{lem}

\begin{lem}[Weak Harnack inequality]\label{lem:weak Harnack}
    Let $u\in C^2(\Omega)$ be a convex function satisfying \eqref{condition} and $v\in W^{1,2}(\Omega)$ be a weak nonnegative supersolution to 
    \[-U^{ij}D_{ij}v=0\] 
    in $\Omega$ with \eqref{condition}. Then there exists $p_0>0$, $C>0$ such that for any section $S_u(x_0,h)\Subset \Omega$, it holds
    $$\left(\fint_{S_u(x_0,h)}v^{p_0}\,\mathrm{d}x\right)^{1/p_0}\le C \inf_{S_u(x_0,h/2)}v,$$ where $C,p_0$ depends only on $n$, $\lambda$, $\Lambda$.
\end{lem}

\begin{thm}[Interior H\"older estimates for homogeneous linearized Monge-Amp\`ere eqautions \cite{CG}]\label{ho Holder}
    Let $u\in C^2(\Omega)$ be a convex function satisfying \eqref{condition} and $S_u(x_0,h_0)\Subset\Omega$. Let $v\in W_{loc}^{2,n}(S_u(x_0,h_0))\cap C(\overline{S_u(x_0,h_0)})$  be a solution to 
    \begin{equation*}
        U^{ij}D_{ij}v=0
    \end{equation*}in $S_u(x_0.h_0)$. Set 
    \[M:=\sup_{x,y\in S_u(x_0,h_0/2)}\left|Du(x)-Du(y)\right|.\] 
    For any $p>0$, there exists constants $\alpha=\alpha(n,\lambda,\Lambda)\in (0,1)$ and $C=C(n,p,\lambda,\Lambda)>0$ such that, for any $x,y\in S_u(x_0,h_0/2)$, \begin{equation*}
    \left|v(x)-v(y)\right|\le CM^{\alpha}h_0^{-\alpha}\Vert v\Vert_{L^{p}(S_u(x_0,h_0))}\left|x-y\right|^{\alpha}.    
    \end{equation*}
\end{thm}


\subsection{H\"older continuity by growth of local integrals}
In the later sections, we will obtain the interior H\"older estimate of \eqref{eq: +- form} by Campanato type estimates on the growth of local integrals. To fit into the setting of Monge-Amp\`ere equation, we need to replace the
balls by sections of $u$.

\begin{thm}\label{thm: Campanato}
Let $\Omega\subset\mathbb{R}^n$ be a bounded convex domain and $u\in C^2(\Omega)$ be a convex function satisfying \eqref{condition} in $\Omega$. 
Suppose that there exist universal $M>0$ and $\alpha\in (0,1)$ such that $v\in L^1(\Omega)$ satisfies 
\[\int_{S_u(x,h)}|v(z)-v_{x,h}|\,\mathrm{d}z\le Mh^{\frac{n+\alpha}{2}}.	\]
Then $v$ is H\"older continuous. Furthermore, there exists $\gamma=\gamma(n,\lambda,\Lambda,\alpha)\in (0,1)$ such that for any $S_u(x_0,h_0)\subset S_u(x_0,4h_0)\Subset \Omega$ there holds \[\sup_{S_u(x_0,h_0)}|v|+\sup_{x\ne y\in S_u(x_0,h_0)}\frac{|v(x)-v(y)|}{|x-y|^{\gamma}}\le C\big(M+\Vert v\Vert_{L^{\infty}(\Omega)}\big)\] 
where $C=C(n,\alpha,\lambda,\Lambda,\Omega,h_0,\operatorname{diam}(S_u(x_0,4h_0)))>0$.
\end{thm}

\begin{proof}
The proof is similar to that of \cite[Theorem 3.1]{HL} or \cite[Appendix A (H3)]{FR}. Since we work with sections instead of Euclidean balls, we will use certain properties of sections associated with strictly convex functions satisfying condition \eqref{condition}.

Let $S_u(x_0,4h_0)\Subset\Omega$. For any $x\in S_u(x_0,h_0)$, by \cite[Theorem 5.13(iii)]{Le24} there exists $\delta_0=\delta_0(n,\lambda,\Lambda)>0$ such that $S_u(x,\delta_0h_0)\subset S_u(x_0,2h_0)\Subset\Omega$. Let $0<h_1<h_2\le \delta h_0$. Then for any $x\in S_u(x_0,h_0)$, we have $S_u(x,h_1),S_u(x,h_2)\Subset\Omega$. Note that 
	\[|v_{x,h_1}-v_{x,h_2}|\le |v(y)-v_{x,h_1}|+|v(y)-v_{x,h_2}|.\] 
	By integration with respect to $y$ in $S_{u}(x,h_1)$, we obtain 
	\begin{align}
		|v_{x,h_1}-v_{x,h_2}|&\le \frac{1}{|S_u(x,h_1)|}\bigg(\int_{S_u(x,h_1)}|v-v_{x,h_1}|+\int_{S_u(x,h_2)}|v-v_{x,h_2}|\bigg)\nonumber\\[4pt]
		&\le C(n)\Lambda^{1/2}h_1^{-n/2}\bigg(\int_{S_u(x,h_1)}|v-v_{x,h_1}|+\int_{S_u(x,h_2)}|v-v_{x,h_2}|\bigg)\label{eq:app1}\\[4pt]
		&\le C(n)M\Lambda^{1/2}h_1^{-n/2}\left(h_1^{\frac{n+\alpha}{2}}+h_2^{\frac{n+\alpha}{2}}\right).\nonumber
	\end{align}  
	
For any $h\le \delta h_0$, with $h_1=h/2^{i+1}$, $h_2=h/2^i$, we obtain 
\[|v_{x,2^{-(i+1)}h}-v_{x,2^{-i}h}|\le C(n,\Lambda)Mh^{\frac{\alpha}{2}}2^{-\frac{\alpha}{2}(i+1)}\] and therefore for $j<k$\[|v_{x,2^{-j}h}-v_{x,2^{-k}h}|\le C(n,\Lambda)Mh^{\frac{\alpha}{2}}\sum_{i=j}^{k-1}2^{-\frac{\alpha}{2}(i+1)}\le C(n,\Lambda,\alpha)Mh^{\frac{\alpha}{2}}2^{-j\alpha/2}.\] 
Then $\{v_{x,2^{-i}h}\}\subset\mathbb{R}$ is a Cauchy sequence, 
hence a convergent one. Its limit $\hat{v}(x_0)$ is independent of the choice of $h$, since \eqref{eq:app1} can be applied with $h_1=2^{-i}h$ and $h_2=2^{-i}h'$ whenever $0<h<h'\le \delta h_0$. Hence we obtain 
\[\hat{v}(x)=\lim\limits_{h\to 0}v_{x,h}\] 
with 
\begin{equation}
		|v_{x,h}-\hat{v}(x)|\le C(n,\Lambda,\alpha)Mh^{\alpha/2}\label{eq:app2}
	\end{equation} 
for any $0<h\le\delta h_0$.
	
Since $v_{x,h}\to v$ in $L^1(\Omega)$ as $h\to0^+$, by the Lebesgue theorem, $v=\hat{v}$ a.e. and \eqref{eq:app2} yields $v_{x,h}\to v(x)$ uniformly in $S_u(x_0,h_0)$. Since $x\mapsto v_{x,h}$ is continuous for any $h>0$, $v(x)$ is continuous. By \eqref{eq:app2} we have
\[|v(x)|\le CMh^{\alpha/2}+|v_{x,h}|\]
 for any $x\in S_u(x_0,h_0)$ and $h\le \delta h_0$. Hence $v$ is bounded in $S_u(x_0,h_0)$ with estimate 
\[\sup_{S_u(x_0,h_0)}|v|\le C\big(Mh_0^{\alpha/2}+\Vert v\Vert_{L^{\infty}(\Omega)}\big).\]
	
Finally, we prove that $v$ is H\"older continuous. Let $x_1,x_2\in S_u(x_0,h_0)$ such that $x_1\in \partial S_u(x_2,h)$ where $h\le \delta h_0$. By \cite[Theorem 5.31]{Le24}, there exists 
$\eta=\eta(n,\lambda,\Lambda)>0$ and $z\in S_u(x_1,h)$ such that 
\begin{equation}
S_u(z,\eta h)\subset S_u(x_1,h)\cap S_u(x_2,h).\label{eq:app3}
\end{equation}   
Then we have \[|v(x_1)-v(x_2)|\le |v(x_1)-v_{x_1,h}|+|v(x_2)-v_{x_2,h}|+|v_{x_1,h}-v_{x_2,h}|.\] 
The first two terms on the right side are estimated in \eqref{eq:app2}. For the last term we write 
\[|v_{x_1,h}-v_{x_2,h}|\le |v_{x_1,h}-v(\xi)|+|v_{x_2,h}-v(\xi)|\] 
and integrating with respect to $\xi$ over $S_u(x_1,h)\cap S_u(x_2,h)$, 
by \eqref{eq:app3} we have 
\begin{align*}
		|v_{x_1,h}-v_{x_2,h}|&\le \frac{1}{|S_u(z,\eta h)|}\bigg(\int_{S_u(x_1,h)}|v-v_{x_1,h}|+\int_{S_u(x_2,h)}|v-v_{x_2,h}|\bigg)\\[4pt]
		&\le C(n,\lambda,\Lambda)h^{-n/2}\bigg(\int_{S_u(x_1,h)}|v-v_{x_1,h}|+\int_{S_u(x_2,h)}|v-v_{x_2,h}|\bigg)\\[4pt]
		&\le C(n,\lambda,\Lambda)Mh^{\alpha/2}.
	\end{align*}

By Caffarelli's pointwise $C^{1,\alpha}$ estimate (see \cite[Theorem 5.18]{Le24}), there exists $\beta=\beta(n,\lambda,\Lambda)\in (0,1)$ such that \[h=u(x_1)-u(x_2)-Du(x_2)(x_1-x_2)\le 4C_1^{-(1+\beta)}|x_1-x_2|^{1+\beta}\] for some $C_1=C_1(n,\lambda,\Lambda,\operatorname{diam}(S_u(x_0,4h_0)))>0$.
Therefore, we have 
\begin{equation*}
|v(x_1)-v(x_2)|\le C(n,\lambda,\Lambda,\alpha)Mh^{\alpha/2}\le C(n,\lambda,\Lambda,\alpha)M|x_1-x_2|^{\gamma},
\end{equation*}
where $\gamma=\frac{\alpha}{2}(1+\beta)\in (0,1)$.
    
For $x_1\notin S_u(x_2,\delta h_0)$, we have 
\[\delta h_0\le u(x_1)-u(x_2)-Du(x_2)(x_1-x_2)\le C_2|x_1-x_2|^{1+\beta}. \]
 Then 
 \begin{align*}
        |v(x_1)-v(x_2)|&\le 2\sup_{S_u(x_0,h_0)}|v|\le C\big(M+h_0^{-\alpha/2}\Vert v\Vert_{L^{\infty}(\Omega)}\big)\cdot(\delta h_0)^{\alpha/2}\\
        &\le  C\big(M+h_0^{-\alpha/2}\Vert v\Vert_{L^{\infty}(\Omega)}\big)|x_1-x_2|^{\gamma},
    \end{align*}
  where $\gamma=\frac{\alpha}{2}(1+\beta)\in (0,1)$. This completes the proof.
\end{proof}
\subsection{Riesz potential}
Given a nonnegative Radon measure $\mu$ in $\Omega$ and $x_0\in\Omega$, we define the {\it (truncated) Riesz potential} with respect to $u$ as
    \begin{equation}
	I_u^{\mu}(x_0,h_0)=\int_0^{h_0} \frac{\mu(S_u(x_0,h))}{h^{\frac{n}{2}-1}}\frac{\mathrm{d}h}{h},\ \ 
\text{for }h_0>0\text{ such that }S_u(x_0,h_0)\Subset\Omega.	\label{eq:Riesz potential}
\end{equation}
The definition \eqref{eq:Riesz potential} is reasonable, but it is a bit different from the classical Riesz potential. For example, when $u=\frac{1}{2}|x|^2$, we know $S_u(x_0,h)=B(x_0,\sqrt{2h})$, and we have \begin{align*}
    I_u^{\mu}(x_0,h_0)&=\int_0^{h_0}\frac{\mu(S_u(x_0,h))}{h^{\frac{n}{2}-1}}\frac{\mathrm{d}h}{h}=2^{n/2}\int_0^{\sqrt{2h_0}}\frac{\mu (B(x_0,s))}{s^{n-2}}\frac{\mathrm{d}s}{s}=2^{n/2}I^{\mu}(x_0,\sqrt{2h_0}),
\end{align*} where we use a change of variable $h=\frac{1}{2}s^2$.
The following element property of the Riesz potential will also be used.
\begin{lem}\label{lem:sim of Riesz potential}
    Let $h_m=2^{-m}h_0$ for $m=0,1,\cdots$. Then there exists $C>0$ depending only on $n$ such that $$C^{-1}I_u^{\mu}(x_0,h_0)\le \sum_{m=0}^{\infty}h_m^{1-n/2}\mu(\overline{S_u(x_0,h_m)})\le CI_u^{\mu}(x_0,2h_0).$$
\end{lem}

\vskip 7pt

\section{potential estimate}\label{sec:potential}

\vskip 7pt

In this section, we prove the potential estimate, which can yield the $L^p$ to $L^{\infty}$ estimate. The proof is based on the Poisson modification technique developed by Trudinger and Wang \cite{TW1}, combined with an iteration method used by Kilpel\"ainen and Mal\'y \cite{KM}, both of which deal with quasilinear elliptic equations. In particular,  the Monge-Amp\`ere Sobolev inequality will be used in our proof. Throughout this section, we still assume $\Omega$ is a bounded domain in $\mathbb{R}^n$ with smooth boundary and $u\in C^2(\Omega)$ be a convex function satisfying \eqref{condition}. 

\begin{lem}\label{lem:key lemma}
    Let $u\in C^2(\Omega)$ be a convex function satisfying \eqref{condition}, $v\in W^{1,2}(\Omega)$ be a weak solution to \eqref{eq: div LMA} and $S_u(x_0,h_0)\Subset\Omega$. For any $p>0$, there exist $C_1>0$ depending only on $n$, $\lambda$, $\Lambda$, $\gamma$ and $C_2>0$ depending only on $n$, $p$, $\Lambda$, such that \begin{equation*}
        \frac{1}{\left|S_u(x_0,h_0/2)\right|^{1/\sigma}}\Vert v_{\pm}\Vert_{L^{\sigma,\infty}(S_u(x_0,h_0/2))}\le C_1 \left(\fint_{A}  v_{\pm}^{p}\,\mathrm{d}x\right)^{\frac{1}{p}}+C_2\left(h_0^{1-n/2}\mu_{\pm}(\overline{S_u(x_0,h_0)})\right),
    \end{equation*}
    where $\sigma=\frac{n}{n-2}$, $A=S_u(x_0,h_0)\setminus\overline{S_u(x_0,h_0/2)}$.
\end{lem}
\begin{proof}
The proof of is inspired by \cite{Ha}, in which a potential estimate for quasilinear elliptic equations with signed measure data was established. 
    We only show the estimate for consider $v_{+}$, since the case of $v_-$ is similar. 
    
    We modify $v$ by $w$ such that $w\in W_{\text{loc}}^{1,2}(\Omega)$ satisfy 
     \begin{equation*}
        \left\{
        \begin{aligned}
            -D_j(U^{ij}D_iw)&=-\mu_{-}&&\text{in }A\\[4pt]
            w&=v&&\text{in }\Omega\setminus A,
        \end{aligned}\right.
    \end{equation*}
   where   $$A:=S_u(x_0,h_0)\setminus\overline{S_u(x_0,h_0/2)}.$$
  
    Since $v-w\in W_0^{1,2}(A)$, for any $k>0$, we have 
    $$\eta:=\min\{\max\{v-w,-k\},0\}\in W_0^{1,2}(A)\cap L^{\infty}(A).$$ 
    Then 
    \begin{align*}
        0&\ge \langle\mu_+,\eta\rangle_A:=\int_A \eta \, \mathrm{d}\mu_+=\langle\mu_+-\mu_-+\mu_-,\eta\rangle_A\\[4pt]
        &=\int_{\{x\in A: -k<v(x)-w(x)<0\}}U^{ij}D_i(v-w)D_j(v-w)\,\mathrm{d}x\ge 0,
    \end{align*}
    which implies 
    \[\left|\{x\in A: -k<v(x)-w(x)<0\}\right|=0\] 
    for any $k>0$. Then we have $w(x)\le v(x)$ for a.e. $x\in A$.

    Now we claim that for any $\varphi\in W_0^{1,2}(S_u(x_0,3h_0/4))$, $0\le \varphi\le 1 $, we have   \begin{equation}
    	\int_{S_u(x_0,3h_0/4)}U^{ij}D_iwD_j\varphi \,\mathrm{d}x\le 2\mu_+(\overline{S_u(x_0,h_0)}).\label{claim:upper bdd}
    \end{equation} Indeed, we consider \[\phi:=H_k(v-w)\varphi\in W_0^{1,2}(A)\cap L^{\infty}(A),\] where $H_k(t)$ is truncated function defined by\[H_k(t):=\frac{1}{k}\min\{\max\{t,-k\},k\}\quad \text{for }k>0.\]
    Then we have \begin{align*}
        0&\le \langle \mu_+,\phi\rangle_A= \langle \mu_+-\mu_-+\mu_-,\phi\rangle_A\\
        &=\int_A U^{ij}D_i(v-w)D_j(H_k(v-w)\varphi)\,\mathrm{d}x\\
        &=\int_A U^{ij}D_i(v-w)D_j(H_k(v-w))\varphi\,\mathrm{d}x+\int_A U^{ij}D_i(v-w)H_k(v-w)D_j\varphi \,\mathrm{d}x.
    \end{align*}
    So 
    \begin{align*}
    	\int_A U^{ij}D_i(w-v)&H_k(v-w)D_j\varphi \,\mathrm{d}x
       \le  \int_A U^{ij}D_i(v-w)D_j(H_k(v-w))\varphi \,\mathrm{d}x\\
       &=\frac{1}{k}\int_{\{x\in A: v(x)-w(x)<k\}} U^{ij}D_i(v-w)D_j(v-w)\varphi\,\mathrm{d}x\\[5pt]
        &\le \frac{1}{k}\int_{\{x\in A: v(x)-w(x)<k\}} U^{ij}D_i(v-w)D_j(v-w)\,\mathrm{d}x\\[5pt]
        &=\langle \mu_+,H_k(v-w)\rangle_A\le \mu_+(A).
    \end{align*}
Note that for any $t>0$, $H_k(t)\to 1$ as $k\to 0$.    Letting $k\to 0$, by Lebesgue dominated convergence theorem, we obtain 
    \begin{eqnarray*}
    \int_{S_u(x_0,3h_0/4)}U^{ij}D_i(w-v)D_j\varphi\,\mathrm{d}x&\le& \int_A U^{ij}D_i(w-v)D_j\varphi \,\mathrm{d}x\\
    &\le& \mu_+(A)\le \mu_+(\overline{S_u(x_0,h_0)}).
    \end{eqnarray*} 
    Here we used $w=v$ in $\Omega\setminus A$, $A\subset \overline{S_u(x_0,h_0)}$ and $\varphi\in W_0^{1,2}(S_u(x_0,3h_0/4))$ .
   Combining the fact
    $$\int_{S_u(x_0,h_0)}U^{ij}D_ivD_j\varphi\,\mathrm{d}x=\langle\mu,\varphi\rangle_{S_u(x_0,3h_0/4)}\le \mu_+(S_u(x_0,3h_0/4))\le \mu_+(\overline{S_u(x_0,h_0)}),$$
    we obtain \eqref{claim:upper bdd} in our claim.

    Next we take the truncated function \[	\psi=\frac{1}{k}\min\{(w_+-l)_+,k\}\in W^{1,2}_0(S_u(x_0,3h_0/4))\quad \text{for }k>0, \] where \[l=\sup_{\partial S_u(x_0,3h_0/4)}w_+:=\inf\left\{l\in\mathbb{R}: (w_+-l)_+\in W_0^{1,2}(S_u(x_0,3h_0/4))\right\}.\] 
    It is clear that $0\leq \psi\leq 1$.
    Applying \eqref{claim:upper bdd}, we have 
    \begin{eqnarray*}
       &&\frac{1}{k}\int_{S_u(x_0,3h_0/4)}U^{ij}D_i(\min\{(w_+-l)_+,k\})D_j(\min\{(w_+-l)_+,k\})\,\mathrm{d}x
       \nonumber\\[4pt]
        &=&\frac{1}{k}\int_{S_u(x_0,3h_0/4)}U^{ij}D_iwD_j(\min\{(w_+-l)_+,k\})\,\mathrm{d}x\\[4pt]
        &=&\int_{S_u(x_0,3h_0/4)}U^{ij}D_iwD_j\psi \,\mathrm{d}x\le 2\mu_+(\overline{S_u(x_0,h_0)}).\nonumber
    \end{eqnarray*}
    By Lemma \ref{lem:Sobolev inequality} and definition of Lorentz norm, we have 
    \begin{align*}
        k|\{&x\in S_u(x_0,3h_0/4): (w_+-l)_+\ge k\}|^{\frac{n-2}{2n}}\\[4pt]&\leq \Vert \min\{(w_+-l)_+,k\}\Vert_{L^{\frac{2n}{n-2},\infty}(S_u(x_0,3h_0/4))}\\[4pt]
        &\leq C\left(\int_{S_u(x_0,3h_0/4)}U^{ij}D_i(\min\{(w_+-l)_+,k\})D_j(\min\{(w_+-l)_+,k\})\,\mathrm{d}x\right)^{1/2}.
    \end{align*} 
    Therefore we obtain $$k|\{x\in S_u(x_0,3h_0/4): (w_+-l)_+\ge k\}|^{\frac{n-2}{n}}\le C\mu_+(\overline{S_u(x_0,h_0)}).$$  Taking the supremum over $k>0$, we have\[\Vert(w_+-l)_+\Vert_{L^{\sigma,\infty}(S_u(x_0,h_0/2))}\le C\mu_+(\overline{S_u(x_0,h_0)}),\quad \text{where }\sigma=\frac{n}{n-2}.\]  
    Note that $0\le w_+\le (w_+-l)_++l$. By \eqref{property 1}, \eqref{property 2} and Lemma \ref{lem:local boundedness}, we have \begin{align*}
        \Vert w_+\Vert_{L^{\sigma,\infty}(S_u(x_0,h_0/2))}&\le \Vert(w_+-l)_++l\Vert_{L^{\sigma,\infty}(S_u(x_0,h_0/2))}\\
        & \le C\left(\Vert(w_+-l)_+\Vert_{L^{\sigma,\infty}(S_u(x_0,h_0/2))}+h_0^{\frac{n-2}{2}}l \right)\\[4pt]
        &\le C\Vert(w_+-l)_+\Vert_{L^{\sigma,\infty}(S_u(x_0,h_0/2))}+h_0^{\frac{n-2}{2}}C_1\left(\fint_{A}w_+^{p}\,\mathrm{d}x\right)^{1/p}\\[4pt]
        &\le C_2\mu_+(\overline{S_u(x_0,h_0)})+ C_1 h_0^{\frac{n-2}{2}}\left(\fint_{A}w_+^{p}\,\mathrm{d}x\right)^{1/p}.
    \end{align*}
    Since $v=w$ in $S_u(x_0,h_0/2)$ and $w\le v$ a.e. in $A$, we obtain the desired result.
\end{proof}

With the estimate in the lemma above, we can obtain Theorem  \ref{thm:potential est} by an iteration method in \cite{KM}.

\begin{proof}[Proof of Theorem \ref{thm:potential est}]
The proof is the same as in \cite{Ha}. For completeness, we still present it.
    It suffices to prove that \eqref{eq:potential est} holds for $v_+$, since $v_-$ is similar. Let $\theta\in(0,1)$ be a sufficiently small constant to be chosen later. For $m=0,1,2,\cdots$, we take $h_m=2^{-m}h_0$ and denote $S_m=S_u(x_0,h_m)$. Let $l_0=0$, and $$l_{m+1}=l_m+\frac{1}{(\theta|S_{m+1}|)^{1/\sigma}}\Vert(v-l_m)_+\Vert_{L^{\sigma,\infty}(S_{m+1})},$$ where $\sigma=\frac{n}{n-2}$. We have \begin{equation*}
        \Vert(v-l_{m-1})_+\Vert_{L^{\sigma,\infty}(S_{m})}=(l_m-l_{m-1})(\theta|S_m|)^{1/\sigma}\quad \text{for all }m\ge 1.
    \end{equation*}

    Now we claim that for any $m\ge 1$, there exists $C>0$ depending only on $n,\lambda$ and $\Lambda$ such that\begin{equation}
        l_{m+1}-l_m\leq C\theta^{2/n}(l_m-l_{m-1})+\frac{C}{\theta^{1/\sigma}}h_m^{1-n/2}\mu_+(\overline{S_m}).\label{claim1}
    \end{equation}
     If the claim \eqref{claim1} holds, we can take $\theta$ such that $C\theta^{2/n}<\frac{1}{2}$ and sum over $m=1,2,\cdots, M$, i.e. $$l_{M+1}-l_1\le \frac{1}{2}l_M+C\sum_{m=1}^Mh_m^{1-n/2}\mu_+(\overline{S_m}).$$ 
     By Lemma \ref{lem:key lemma}, we have \begin{align*}
         l_1=\frac{1}{\theta^{1/\sigma}}\frac{1}{|S_1|^{1/\sigma}}\Vert v_+\Vert_{L^{\sigma,\infty}(S_1)}\le \frac{1}{\theta^{1/\sigma}}\left[C\left(\fint_{S_0-\overline{S_1}}v_+^{p}\,\mathrm{d}x\right)^{1/p}+h_0^{1-n/2}\mu_+(\overline{S_0})\right]. 
     \end{align*} 
      By Lemma \ref{lem:sim of Riesz potential},  $l_m$ is bounded as $m\to\infty$ and converges to 
     $$l_{\infty}\le C\left(\fint_{S_0-\overline{S_1}}v_+^{p}\,\mathrm{d}x\right)^{1/p}+C \sum_{m=0}^{\infty}h_m^{1-n/2}\mu_+(\overline{S_m}).$$
         By definition of $l_m$, we have $$\frac{1}{|S_m|^{1/\sigma}}\Vert(v-l_{m-1})_+\Vert_{L^{\sigma,\infty}(S_m)}\le \theta^{1/\sigma}(l_m-l_{m-1})\to 0,$$ as $m\to \infty$.
     Note that \begin{eqnarray*}
         &&(l_m-l_{m-1})\left|\{x\in S_m: v(x)\ge l_m\}\right|^{1/\sigma}\\
         &=&(l_m-l_{m-1})\left|\{x\in S_m: (v(x)-l_{m-1})_+\ge l_m-l_{m-1}\}\right|^{1/\sigma}\\[4pt]
         &\le& \Vert(v-l_{m-1})_+\Vert_{L^{\sigma,\infty}(S_m)}=(l_m-l_{m-1})(\theta|S_m|)^{1/\sigma},
     \end{eqnarray*}
     where we may assume $l_m>l_{m-1}$. 
     So $$\left|\{x\in S_m: v(x)\ge l_m\}\right|^{1/\sigma}\le (\theta|S_m|)^{1/\sigma}.$$ 
     By H\"older inequality \eqref{property 3: Holder ineq} with $q=q_1=1$ and $q_2=\infty$, we have 
     \begin{align*}
         \int_{S_m}(v-l_m)_+\,\mathrm{d}x&\le \Vert \chi_{\{v\ge l_m\}}\Vert_{L^{n/2,1}(S_m)}\Vert(v-l_m)_+
         \Vert_{L^{\sigma,\infty}(S_m)}\\[4pt]
         &\le C \left|\{x\in S_m: v(x)\ge l_m\}\right|^{2/n}\cdot\Vert(v-l_m)_+\Vert_{L^{\sigma,\infty}(S_m)}\\[4pt]
         &\le (\theta|S_m|)^{2/n}\Vert(v-l_{m-1})_+\Vert_{L^{\sigma,\infty}(S_m)}.
     \end{align*} 
     So we obtain 
     \begin{equation}
       \fint_{S_m}v_+\,\mathrm{d}x -l_m\leq   \fint_{S_m}(v-l_m)_+\,\mathrm{d}x\le C\frac{1}{|S_m|^{1/\sigma}}\Vert(v-l_{m-1})_+\Vert_{L^{\sigma,\infty}(S_m)}\to 0,\label{local bdd}
     \end{equation}as $m\to \infty$. Hence we have 
     $$v_+(x_0)=\lim_{m\to\infty}\fint_{S_m}v_+\,\mathrm{d}x\le l_{\infty}\le C\left(\fint_{S_u(x_0,h_0)-\overline{S_u(x_0,h_0/2)}}v_+^{p}\,\mathrm{d}x\right)^{1/p}+CI_u^{\mu_+}(x_0,2h_0).$$

     Now it remains to show that \eqref{claim1} holds. We may assume $l_m>l_{m-1}$, otherwise $l_{m+1}=l_m$ and \eqref{claim1} holds obviously. 
     Applying  Lemma \ref{lem:key lemma} for $v-l_m$ with $\gamma=1$, by \eqref{local bdd}, we obtain \begin{align*}
         (l_{m+1}-l_m)\theta^{1/\sigma}&=\frac{1}{|S_{m+1}|^{1/\sigma}}\Vert(v-l_m)_+\Vert_{L^{\sigma,\infty}(S_{m+1})}\\[4pt]
         &\le C\fint_{S_m-\overline{S_{m+1}}}(v-l_m)_+\,\mathrm{d}x+ Ch_m^{1-n/2}\mu_+(\overline{S_m})\\[4pt]
         &\le C\left(\frac{\theta^{2/n}}{|S_m|^{1/\sigma}}\Vert(v-l_{m-1})_+\Vert_{L^{\sigma,\infty}(S_m)}+h_m^{1-n/2}\mu_+(\overline{S_m})\right)\\[4pt]
         &=C\theta(l_m-l_{m-1})+Ch_m^{1-n/2}\mu_+(\overline{S_m}).
     \end{align*}Then \eqref{claim1} follows.
\end{proof}

\begin{rem}
We make the following remarks regarding possible extensions of Theorem \ref{thm:potential est}:
\begin{enumerate}
    \item[(i)] The key ingredient in the proof of \eqref{eq:potential est} is the Sobolev inequality \eqref{eq:Sobolev}, which also holds for the complex linearized Monge-Amp\`ere equations; see \cite{WZ24}. Therefore, the potential estimate \eqref{eq:potential est} can be established in the complex setting by a similar argument.

    \item[(ii)] The Sobolev inequality \eqref{eq:Sobolev} is known to hold under the so-called doubling property of the Monge-Amp\`ere measure $\mu_u$ (see \cite{Ma1}), namely, there exist constants $C > 0$ and $\alpha \in (0,1)$ such that
    \begin{equation}\label{eq: doubling property}
        \mu_u(S_u(x,t)) \leq C \mu_u(\alpha S_u(x,t)),
    \end{equation}
    for every section $S_u(x,t)$, where $\alpha S(x,t)$ denotes the $\alpha$-dilation of the set $S(x,t)$ with respect to its center of mass. As a result, the structural condition \eqref{condition} on $u$ in Theorem \ref{thm:potential est} can be replaced by the doubling condition.
\end{enumerate}
\end{rem}
 Combining the potential estimate \eqref{eq:potential est}, one can obtain the $L^p$ to $L^{\infty}$ estimate below:
\begin{cor}\label{cor:Lp to Linfty est}
	Let $u\in C^2(\Omega)$ be a convex function satisfying \eqref{condition} and $v\in W^{1,2}(\Omega)$ be a weak solution to \eqref{eq: div LMA} in $\Omega$. Suppose that there exist $M>0$ and $\varepsilon>0$ such that 
	\begin{equation}\label{condition: growth of measure}
		|\mu|(S_u(x,h))\le M h^{\frac{n}{2}-1+\varepsilon}
	\end{equation} whenever $x\in \Omega$ such that $S_u(x,h)\subset\Omega$. Then for any $S_u(x_0,2h_0)\Subset\Omega$ and any $p>0$, there exists $C>0$ depending only on $n,\lambda,\Lambda,p,h_0$ such that 
    \begin{equation*}
		\Vert v\Vert_{L^{\infty}(S_u(x_0, h_0))}\le C\left(\Vert v\Vert_{L^p(S_u(x_0,2h_0)))}+Mh_0^{\varepsilon}\right)
	\end{equation*}
\end{cor}
\begin{proof}
	For any $x\in S_u(x_0,h_0)$, by \cite[Theorem 5.13(iii)]{Le24}, there exists constant $\delta_0=\delta_0(n,\lambda,\Lambda)>0$ such that $S_u(x,\delta_0h_0)\subset S_u(x_0,3h_0/2)\Subset\Omega$. Then for any Lebesgue point $y\in S_u(x_0,h_0)$, we have $S_u(y,\delta_0 h_0)\subset S_u(x_0,3h_0/2) $. Applying \eqref{eq:potential est} we have \begin{align*}
		|v(y)|&=v_+(y)+v_-(y)\\
		&\le C\left(\fint_{S_u(y,\delta_0 h_0)}v_+^p\,\mathrm{d}x\right)^{1/p}+C\left(\fint_{S_u(y,\delta_0h_0)}v_-^p\,\mathrm{d}x\right)^{1/p}+CI_u^{|\mu|}(y,2\delta_0h_0)\\
		&\le 2C\left(\fint_{S_u(x_0,3h_0/2)}(v_++v_-)^p\,\mathrm{d}x\right)^{1/p}+CMh_0^{\varepsilon}\\
		&\le C\left(h_0^{-\frac{n}{2p}}\Vert v\Vert_{L^p(S_u(x_0,2h_0))}+Mh_0^{\varepsilon}\right).
	\end{align*} 
	Then \[\Vert v\Vert_{L^{\infty}(S_u(x_0,h_0))}\le C\Vert v\Vert_{L^p(S_u(x_0,2h_0))}+CMh_0^{\varepsilon}.\]
	\end{proof}


\vskip 7pt
 \section{H\"older regularity}\label{sec:Holder1}

 In this section, we present proofs for the H\"older regularity of solutions to \eqref{eq: div LMA} and obtain Theorem~\ref{thm:divF+f}.
 \subsection{H\"older regularity when the right-hand side is a non-negative Radon measure} 
We first provide a proof of H\"older regularity for non-negative data by directly applying the potential estimate together with the Harnack inequality for the homogeneous equation. 

\begin{thm}\label{thm:Holder by potential est for +}
	Let $u\in C^2(\Omega)$ be a convex function satisfying \eqref{condition} and $\mu$ be a non-negative Radon measure. Suppose that there exist $M>0$ and $\varepsilon>0$ such that 
	\begin{equation*}
		\mu(S_u(x,h))\le M h^{\frac{n}{2}-1+\varepsilon}
	\end{equation*} whenever $x\in \Omega$ such that $S_u(x,h)\subset\Omega$. Then the solution $v$ 
	to \begin{equation}
	    -U^{ij}D_{ij}v=\mu\label{split 1}
	\end{equation} is H\"older continuous. Moreover, for any given section $S_u(x_0,2h_0)\Subset\Omega$, there exist constants $\gamma\in(0,1)$ depending only on $n$, $\varepsilon$, $\lambda$, $\Lambda$ and  $C>0$ depending only on $n$, $\varepsilon$, $M$, $\lambda$, $\Lambda$, $h_0$, $\operatorname{diam}(S_u(x_0,2h_0))$, $\Vert v\Vert_{L^{\infty}(\Omega)}$  such that \begin{equation*}
		|v(x)-v(y)|\leq C|x-y|^\gamma.
	\end{equation*}for any $x,y\in S_u(x_0,h_0)$. 
\end{thm}

\begin{proof}
     For any $0<h\le 2h_0$, we denote \[M_h=\sup_{S_u(x_0,h)}v,\quad \quad m_h=\inf_{S_u(x_0,h)}v.\]  For any $x\in S_u(x_0,h)$, by \cite[Theorem 5.13(iii)]{Le24}, there exists $\delta_0=\delta_0(n,\lambda,\Lambda)>0$ such that $S_u(x_,\delta_0h)\subset S_u(x_0,3h/2)$. Let $h'>h$ be chosen later and apply \eqref{eq:potential est} to the solution $v-m_{h'}$ of \eqref{split 1} in $S_u(x,\delta_0 h)$ to obtain \[v(x)-m_{h'} \le C\left[\fint_{S_u(x,\delta_0 h)}(v(z)-m_{h'})^{p}\,\mathrm{d}z\right]^{1/p}+CI^{\mu}(x,2\delta_0h)\] for any $p>0$. By \cite[Theorem 5.28]{Le24}, there exists $\theta=\theta(n,\lambda,\Lambda)>3$ such that \[S_u(x_0,3h/2)\subset S_u(x,\theta h).\] Since $v-m_{h'}$ is a nonnegative supersolution of \[-U^{ij}D_{ij}v=0\] in $S_u(x_0,h')$, where $h':=2\theta h$ and $h$ is chosen such that $S_u(x_0,2\theta h)\Subset\Omega$, we can apply weak Harnack inequality (Lemma \ref{lem:weak Harnack}) to obtain  \begin{align*}
        \left(\fint_{S_u(x,\delta_0 h)}(v(z)-m_{h'})^{p}\,\mathrm{d}z\right)^{1/p}& \le C\left(\fint_{S_u(x,2\theta h)}(v(z)-m_{h'})^{p_0}\,\mathrm{d}z\right)^{1/p_0}\\[4pt]
        &\le C\left(\inf_{S_u(x,\theta h)}v-m_{h'}\right)\\[4pt]
        & \le C\left(\inf_{S_u(x_0,h)}v-m_{h'}\right)=C(m_{h}-m_{h'}),
    \end{align*}where we choose $p=p_0$ as in Lemma \ref{lem:weak Harnack}.

     Noting that \begin{align*}
       I^{\mu}(x,2\delta_0 h)=\int_0^{2\delta_0 h}\frac{\mu(S_u(x,\rho))}{\rho^{n/2-1}}\frac{\mathrm{d}\rho}{\rho}\le M\int_0^{2\delta_0h}\rho^{\varepsilon}\frac{\mathrm{d}\rho}{\rho}=Ch^{\varepsilon},
   \end{align*}we obtain that \[v(x)-m_{h'}\le C(m_h-m_{h'})+Ch^{\varepsilon}\quad \text{ for all }x\in S_u(x_0,h).\] i.e. \[M_h-m_{h'}\le C(m_h-m_{h'})+Ch^{\varepsilon},\] which yields \begin{align*}
       C(M_h-m_h)&\le (C-1)M_h-(C-1)m_{h'}+Ch^{\varepsilon}\\
       &\le (C-1)(M_h-m_{h'})+Ch^{\varepsilon}\\
       &\le (C-1)(M_{h'}-m_{h'})+Ch^{\varepsilon}.
   \end{align*} Hence for all $h\le h_0$, we have
      \begin{equation*}
       \operatorname{osc}_{S_u(x_0,h)}v\le \beta \operatorname{osc}_{S_u(x_0,2\theta h)}v+h^{\varepsilon}.
   \end{equation*} where $\beta=\frac{C-1}{C}<1$.
  Then by the well-known De Giorgi lemma on the iterations of monotone functions on the real
  interval $(0,h_0)$ (see \cite[Lemma 8.23]{GT}), we obtain \begin{equation*}
      \operatorname{osc}_{S_u(x_0,h)}v\le C(n,\lambda,\Lambda,M)\left(\frac{h}{h_0}\right)^{\gamma}(\Vert v\Vert_{L^{\infty}(\Omega)}+h_0^{\varepsilon})
  \end{equation*}
for some $\gamma\in(0,1)$ depending only on $n$, $\varepsilon$, $\lambda$ and $\Lambda$. Then applying the properties of sections (similar with \cite[Theorem 12.14]{Le24}), we can obtain that
  \begin{equation*}
      \left|v(x)-v(y)\right|\le C\left|x-y\right|^{\gamma},  \quad \forall x,y\in S_u(x_0,h_0) 
  \end{equation*}
  where $C>0$ depending only on $n$, $\varepsilon$, $M$, $\lambda$, $\Lambda$, $h_0$, $\operatorname{diam}(S_u(x_0,2h_0))$ and $\Vert v\Vert_{L^{\infty}(\Omega)}$.
  \end{proof}

\vskip 7pt

 \subsection{H\"older regularity when the right-hand side is a signed Radon measure} 
Prior to the proof, we recall a well-known iteration lemma that will be used in the argument.
\begin{lem}[{\cite[Lemma 3.4]{HL}}]\label{lem:iteration}
	Let $\phi(t)$ be a nonnegative and nondecreasing function on $[0,R]$. Suppose that
	$$\phi(\rho)\leq A\left[\left(\frac{\rho}{r}\right)^\alpha+\varepsilon\right]\phi(r)+Br^\beta$$
	for any $0<\rho\leq r\leq R$, with $A$, $B$, $\alpha$, $\beta$ nonnegative constants and $\beta<\alpha$. Then for any $\gamma\in(\beta,\alpha)$, there exists a constant $\varepsilon_0=\varepsilon_0(A,\alpha,\beta,\gamma)$ such that if $\varepsilon<\varepsilon_0$ we have for all $0<\rho\leq r\leq R$
	$$\phi(\rho)\leq c\left[\left(\frac{\rho}{r}\right)^{\gamma}\phi(r)+B\rho^\beta\right]$$
	where $c$ is a constant depending on $A$, $\alpha$, $\beta$, $\gamma$. In particular, for any $0<r\leq R$, we have
	$$\phi(r)\leq c\left(\frac{\phi(R)}{R^\gamma}r^\gamma+Br^\beta\right).$$
\end{lem}

\vskip 10pt

To prove Theorem \ref{thm: Holder for signed data}, we will need the following lemma, which can be regarded as a variant of Harnack’s inequality for the homogeneous equation.
\begin{lem}\label{ite-lemma}
Let $u\in C^2(\Omega)$ be a convex function satisfying \eqref{condition}, and let $w$ be the solution to
 \[-U^{ij}D_{ij}w=0\]  
 in $S_u(x_0,h_0)\Subset\Omega$. Then there exists $\varepsilon'>0$ depending only on $n,\lambda,\Lambda$, and $C>0$ depending only on $n,\lambda,\Lambda,h_0,\operatorname{diam}(\Omega)$, such that
 \begin{equation}\label{eq:Dw}
        \int_{S_u(x_0,\rho)}U^{ij}D_iwD_jw\,\mathrm{d}x\leq C\left(\frac{\rho}{h}\right)^{\frac{n}{2}-1+\varepsilon'}\int_{S_u(x_0,h)}U^{ij}D_iwD_jw\,\mathrm{d}x
    \end{equation}
    for all $0<\rho<h\le h_0$.
\end{lem}
\begin{proof}
By translation and dilation, we only need to consider $x_0=0$ and $h=1$. Restrict to the range $\rho\in(0,1/4]$, since \eqref{eq:Dw} is trivial for $\rho\in (1/4,1)$. We also assume that $u=2\rho$ on $\partial S_u(0,2\rho)$ by subtracting the support function. We may further assume that $\displaystyle\int_{S_u(0,1)}w\,\mathrm{d}x=0$ since the function $w-\displaystyle\fint_{S_u(0,1)}w\,\mathrm{d}x$ still solves $-U^{ij}D_{ij}w=0$. 
    Then by the weak (1,2)-Poincar\'e inequality (Lemma \ref{lem:Poincare}),
    $$\int_{S_u(0,1)}|w|\,\mathrm{d}x\leq C\bigg(\int_{S_u(0,1)}U^{ij}D_iwD_jw\,\mathrm{d}x\bigg)^{1/2}.$$
    Applying Theorem \ref{ho Holder} with $p=1$, there exists $\alpha\in(0,1)$, such that for any $x\in S_{u}(0,1/2)$,
    \[|w(x)-w(0)|^2\leq C|x|^{2\alpha}\cdot \|w\|_{L^1(S_u(0,1))}^2\leq C|x|^{2\alpha}\int_{S_u(0,1)}U^{ij}D_iwD_jw\,\mathrm{d}x.\]
    
    For any $0<\rho\leq 1/4$, we take $\varphi=\zeta^2(w-w(0))$, where $\zeta=2\rho-u$. Then by the Cauchy-Schwarz inequality, we have
    \begin{align*}
        0&=\int_{S_u(0,2\rho)}U^{ij}D_iw D_j\varphi\,\mathrm{d}x\\
        &=\int_{S_u(0,2\rho)}U^{ij}D_iw\cdot (\zeta^2D_jw+2\zeta D_j\zeta(w-w(0)))\,\mathrm{d}x\\[4pt]
        &\geq \int_{S_u(0,2\rho)}\zeta^2U^{ij}D_iwD_jw\,\mathrm{d}x-\frac{1}{2}\int_{S_u(0,2\rho)}\zeta^2U^{ij}D_iwD_jw\,\mathrm{d}x\\
        &\quad\quad\quad -C\int_{S_u(0,2\rho)}U^{ij}D_i\zeta D_j\zeta\cdot (w-w(0))^2\,\mathrm{d}x\\[4pt]
        &\geq \frac{1}{2}\int_{S_u(0,2\rho)}\zeta^2U^{ij}D_iwD_jw\,\mathrm{d}x-C\sup_{S_u(0,2\rho)}|w-w(0)|^2\int_{S_u(0,2\rho)}U^{ij}D_i\zeta D_j\zeta\,\mathrm{d}x.
    \end{align*}
    Note that
    \begin{align*}
        \int_{S_u(0,2\rho)}U^{ij}D_i\zeta D_j\zeta\,\mathrm{d}x&=-\int_{S_u(0,2\rho)}U^{ij}D_{ij}\zeta \,\zeta\,\mathrm{d}x=\int_{S_u(0,2\rho)}U^{ij}D_{ij}u \,\zeta\,\mathrm{d}x\\
        &=\int_{S_u(0,2\rho)}n\det D^2u\cdot \zeta\,\mathrm{d}x\leq C\rho^{n/2+1}.
    \end{align*}
    Hence we have
    \begin{align*}
        \int_{S_u(0,\rho)}U^{ij}D_iwD_jw\,\mathrm{d}x&\leq \frac{c}{\rho^2}\int_{S_u(0,2\rho)}\zeta^2U^{ij}D_iwD_jw\,\mathrm{d}x\\[4pt]
        &\leq C\rho^{n/2-1}\sup_{S_u(0,2\rho)}|w-w(0)|^2\\[4pt]
        &\leq C\rho^{\frac{n}{2}-1+\varepsilon'}\int_{S_u(0,1)}U^{ij}D_iwD_jw\,\mathrm{d}x.
    \end{align*}
    We obtain \eqref{eq:Dw}.
\end{proof}

\begin{proof}[Proof of Theorem~\ref{thm: Holder for signed data} ]
We compare $v$ with solutions to the homogeneous equation.  For all $h\leq h_0$, let $w$ be the solution to
    \begin{equation*}
        \left\{
        \begin{aligned}
            -U^{ij}D_{ij}w&=0&&\text{in  }     S_u(x_0,h),\\[4pt]
            w&=v&&\text{on  }\partial S_u(x_0,h).
        \end{aligned}\right.
    \end{equation*}
    Then we know that
    \begin{align}
        \int_{S_u(x_0,h)}&U^{ij}D_i(v-w)D_j(v-w)\,\mathrm{d}x\nonumber\\&=\int_{S_u(x_0,h)}U^{ij}D_ivD_j(v-w)\,\mathrm{d}x-\int_{S_u(x_0,h)}U^{ij}D_jwD_i(v-w)\,\mathrm{d}x\nonumber\\[4pt]
        &=\int_{S_u(x_0,h)}(v-w)\,\mathrm{d}\mu\label{eq:Dv-w}\\&\leq 2\|v\|_{L^\infty(S_u(x_0,h_0))}|\mu|(S_u(x_0,h))\nonumber\\
        &\leq 2\|v\|_{L^\infty(S_u(x_0,h_0))}M h^{\frac{n}{2}-1+\varepsilon}.\nonumber
    \end{align}
    
    Next, we show that
     \begin{equation}\label{eq:Dw<Dv}
    	\int_{S_u(x_0,h)}U^{ij}D_iwD_jw\,\mathrm{d}x\leq \int_{S_u(x_0,h)}U^{ij}D_ivD_jv\,\mathrm{d}x.
    \end{equation}
   Indeed, since
    $$0=\int_{S_u(x_0,h)}U^{ij}D_iwD_j(v-w)\,\mathrm{d}x,$$
    we obtain
    \begin{align*}
        \int_{S_u(x_0,h)}U^{ij}D_iwD_jw\,\mathrm{d}x&=\int_{S_u(x_0,h)}U^{ij}D_iwD_jv\,\mathrm{d}x\\
        &\leq \frac{1}{2}\int_{S_u(x_0,h)}U^{ij}D_iwD_jw\,\mathrm{d}x+\frac{1}{2}\int_{S_u(x_0,h)}U^{ij}D_ivD_jv\,\mathrm{d}x,
    \end{align*}
    which implies \eqref{eq:Dw<Dv}.
   
Write $v=w+(v-w)$. By \eqref{eq:Dw}, \eqref{eq:Dv-w} and \eqref{eq:Dw<Dv}, we have
    \begin{eqnarray*}
        &&\int_{S_u(x_0,\rho)}U^{ij}D_ivD_jv\,\mathrm{d}x\\
        &\leq& 2\left(\int_{S_u(x_0,\rho)}U^{ij}D_iwD_jw\,\mathrm{d}x+\int_{S_u(x_0,\rho)}U^{ij}D_i(v-w)D_j(v-w)\,\mathrm{d}x\right)\\[4pt]
        &\leq& C\left(\frac{\rho}{h}\right)^{\frac{n}{2}-1+\varepsilon'}\int_{S_u(x_0,h)}U^{ij}D_iwD_jw\,\mathrm{d}x+C\|v\|_{L^\infty(S_u(x_0,h_0))} h^{\frac{n}{2}-1+\varepsilon}\\[4pt]
        &\leq& C\left(\frac{\rho}{h}\right)^{\frac{n}{2}-1+\varepsilon'}\int_{S_u(x_0,h)}U^{ij}D_ivD_jv\,\mathrm{d}x+C\|v\|_{L^\infty(S_u(x_0,h_0))} h^{\frac{n}{2}-1+\varepsilon}.
    \end{eqnarray*}
    Then by Lemma \ref{lem:iteration}, there is
    $$\int_{S_u(x_0,\rho)}U^{ij}D_ivD_jv\,\mathrm{d}x\leq C\|v\|_{L^\infty(S_u(x_0,h_0))}\left(\frac{\rho}{h}\right)^{\frac{n}{2}-1+\varepsilon}$$
    for all $0<\rho<h\leq h_0$.
    By weak (1,2)-Poincar\'e inequality in Lemma \ref{lem:Poincare}, we know
    \begin{eqnarray*}
    \int_{S_u(x_0,\rho)}|v-v_{x_0,\rho}|\,\mathrm{d}x&\leq &C\rho^{\frac{1}{2}+\frac{n}{4}}\bigg(\int_{S_u(x_0,\rho)}U^{ij}D_ivD_jv\,\mathrm{d}x\bigg)^{1/2}\\
    &\leq& C\|v\|^{1/2}_{L^\infty(S_u(x_0,h_0))}\rho^{\frac{n+\varepsilon}{2}}.
    \end{eqnarray*}
By Theorem \ref{thm: Campanato} and Corollary~\ref{cor:Lp to Linfty est},  we have 
\begin{align*}
    |v(x)-v(y)|&\leq C\left(\|v\|_{L^\infty(S_u(x_0,h_0))}+M\right)|x-y|^\gamma\\
    &\le C\left(\|v\|_{L^p(S_u(x_0,h_0))}+M\right)|x-y|^\gamma
\end{align*}
   as desired.
\end{proof}

\subsection{Proof of Theorem~\ref{thm:divF+f}}
 We shall give the proof of Theorem \ref{thm:divF+f}, which depends on Theorem \ref{thm: Holder for signed data} and the following observation.
\begin{lem}\label{lem: area est}
    Let $X\subset\mathbb{R}^n$ be a bounded convex body with smooth boundary and $B_r(x_0)\subset X$ be a ball with radius $r$ in $X$. Then we have \begin{equation}\label{eq: area est}
        |\partial X|\le \frac{n |X|}{r}.
    \end{equation} 
\end{lem}

 \begin{figure}[!htbp]

\tikzset{every picture/.style={line width=0.75pt}} 

\begin{tikzpicture}[x=0.75pt,y=0.75pt,yscale=-1.1,xscale=1.1]

\draw [line width=1]    (126,234) -- (548.33,234) ;
\draw [shift={(551.33,234)}, rotate = 180] [color={rgb, 255:red, 0; green, 0; blue, 0 }  ][line width=1]    (14.21,-4.28) .. controls (9.04,-1.82) and (4.3,-0.39) .. (0,0) .. controls (4.3,0.39) and (9.04,1.82) .. (14.21,4.28)   ;
\draw [line width=1]    (195,290) -- (195,11.17) ;
\draw [shift={(195,8.17)}, rotate = 90] [color={rgb, 255:red, 0; green, 0; blue, 0 }  ][line width=1]    (14.21,-4.28) .. controls (9.04,-1.82) and (4.3,-0.39) .. (0,0) .. controls (4.3,0.39) and (9.04,1.82) .. (14.21,4.28)   ;
\draw  [line width=1]  (152.33,158.87) .. controls (152.33,124.24) and (230.01,96.17) .. (325.83,96.17) .. controls (421.65,96.17) and (499.33,124.24) .. (499.33,158.87) .. controls (499.33,193.51) and (421.65,221.58) .. (325.83,221.58) .. controls (230.01,221.58) and (152.33,193.51) .. (152.33,158.87) -- cycle ;
\draw  [line width=1]  (277.15,158.87) .. controls (277.15,131.99) and (298.94,110.19) .. (325.83,110.19) .. controls (352.72,110.19) and (374.52,131.99) .. (374.52,158.87) .. controls (374.52,185.76) and (352.72,207.56) .. (325.83,207.56) .. controls (298.94,207.56) and (277.15,185.76) .. (277.15,158.87) -- cycle ;
\draw [line width=1]    (325.83,158.87) -- (440.53,115.23) ;
\draw [shift={(443.33,114.17)}, rotate = 159.17] [color={rgb, 255:red, 0; green, 0; blue, 0 }  ][line width=1]    (14.21,-4.28) .. controls (9.04,-1.82) and (4.3,-0.39) .. (0,0) .. controls (4.3,0.39) and (9.04,1.82) .. (14.21,4.28)   ;
\draw [line width=1]    (443.33,114.17) -- (455.07,88.89) ;
\draw [shift={(456.33,86.17)}, rotate = 114.9] [color={rgb, 255:red, 0; green, 0; blue, 0 }  ][line width=1]    (14.21,-4.28) .. controls (9.04,-1.82) and (4.3,-0.39) .. (0,0) .. controls (4.3,0.39) and (9.04,1.82) .. (14.21,4.28)   ;
\draw [line width=1]    (326.33,160.17) -- (346.02,119.86) ;
\draw [shift={(347.33,117.17)}, rotate = 116.03] [color={rgb, 255:red, 0; green, 0; blue, 0 }  ][line width=1]    (14.21,-4.28) .. controls (9.04,-1.82) and (4.3,-0.39) .. (0,0) .. controls (4.3,0.39) and (9.04,1.82) .. (14.21,4.28)   ;
\draw [line width=1]  [dash pattern={on 5.63pt off 4.5pt}]  (349.33,81.17) -- (529.33,141.17) ;
\draw [line width=1]  [dash pattern={on 5.63pt off 4.5pt}]  (347.33,117.17) -- (363.33,84.17) ;

\draw (440,118) node [anchor=north west][inner sep=0.75pt]  [font=\large,xscale=0.75,yscale=0.75] [align=left] {$\displaystyle x$};
\draw (450,70) node [anchor=north west][inner sep=0.75pt]  [font=\large,xscale=0.75,yscale=0.75] [align=left] {$\displaystyle \mathbf{n}_{x}$};
\draw (308,157.33) node [anchor=north west][inner sep=0.75pt]  [font=\large,xscale=0.75,yscale=0.75] [align=left] {$\displaystyle x_{0}$};
\draw (204.67,176.33) node [anchor=north west][inner sep=0.75pt]  [font=\large,xscale=0.75,yscale=0.75] [align=left] {$\displaystyle X$};
\draw (323,130) node [anchor=north west][inner sep=0.75pt]  [font=\large,xscale=0.75,yscale=0.75] [align=left] {$\displaystyle r$};
\end{tikzpicture}
        \caption{Boundary area controlled by volume}
        \label{fig:area}
    \end{figure}
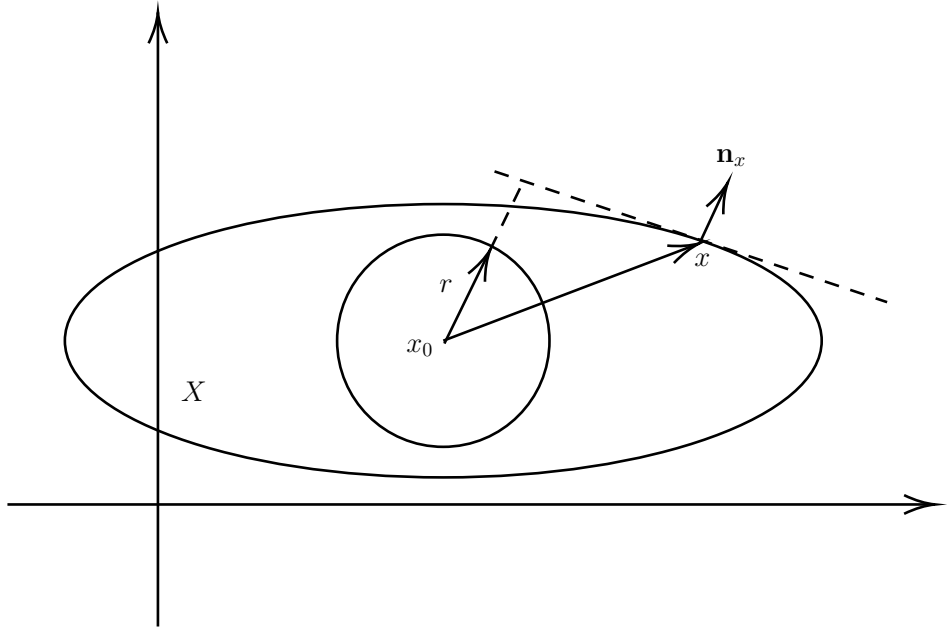
    \begin{proof}
        Note that \[ |X|=\frac{1}{n}\int_{\partial X}(x-x_0)\cdot\mathbf{n}_x\,\mathrm{d}S\ge \frac{1}{n}\int_{\partial X}r\,\mathrm{d}S=\frac{r|\partial X|}{n}, \] which yields \eqref{eq: area est}.
        See Figure \ref{fig:area} for the geometric interpretation.
    \end{proof}

    \begin{proof}[Proof of Theorem \ref{thm:divF+f}] 
    Denote $$\mu_F:=\operatorname{div}\mathbf{F}\,\mathrm{d}x\quad\text{and}\quad\mu_f:=f\,\mathrm{d}x.$$
    Given any section $S_u(x,h)\subset S_u(x_0,h_0)\Subset\Omega $, we claim that there exist universal $M>0$ and $\varepsilon>0$ such that \[ |\mu_F|(S_u(x,h))\le Mh^{\frac{n}{2}-1+\varepsilon}.\]
    By interior $C^{1,\alpha}$ estimates for Monge-Amp\`ere equation, we have  \[ S_u(x, h)\supset B_{ch^{\frac{1}{1+\alpha}}}(x)\] for some $c=c(n,\lambda,\Lambda,h_0,\operatorname{diam}(\Omega))>0$ and $\alpha=\alpha(n,\lambda,\Lambda)\in (0,1)$ (see \cite[Remark 5.23]{Le24}). Then by Lemma \ref{lem: area est}, we have \begin{equation}\label{eq: surface-area est}
        |\partial S_u(x,h)|\le \frac{n |S_u(x,h)|}{h^{\frac{1}{1+\alpha}}}\le Ch^{\frac{n}{2}-\frac{1}{1+\alpha}}=Ch^{\frac{n}{2}-1+\frac{\alpha}{1+\alpha}}.
    \end{equation} 
Since $\div F\ge 0$, by Gauss-Green formula for bounded divergence-measure vector fields that was proven in Chen-Torres \cite{CT} and Chen-Torres-Ziemer \cite{CTZ} (see also \cite{Si}), we have \begin{align*}
    |\mu_F|(S_u(x,h))&=\int_{S_u(x,h)}\div \mathbf{F}\,\mathrm{d}x=\int_{\partial S_u(x,h)}\mathbf{F}\cdot \nu\,\mathrm{d}S\\
    &\le \Vert \mathbf{F}\Vert_{L^{\infty}(\Omega)}|\partial S_u(x,h)|\\
    &\le C\Vert \mathbf{F}\Vert_{L^{\infty}(\Omega)}h^{\frac{n}{2}-1+\frac{\alpha}{1+\alpha}}.
\end{align*}
Combining \eqref{eq: f-growth}, we have \[|\mu|(S_u(x,h))\le |\mu_F|(S_u(x,h))+|\mu_f|(S_u(x,h))\le C\left(\Vert \mathbf{F}\Vert_{L^{\infty}(\Omega)}+\Vert f\Vert_{L^q}\right)h^{\frac{n}{2}-1+\varepsilon},\] where $\varepsilon=\min\left\{\frac{\alpha}{1+\alpha}, 1-\frac{n}{2q}\right\}>0$.
 Then by Theorem~\ref{thm: Holder for signed data}, we have \begin{align*}
     |v(x)-v(y)|&\le C\left(\Vert v\Vert_{L^{p}(S_u(x_0,2h_0))}+\Vert \mathbf{F}\Vert_{L^{\infty}(\Omega)}+\Vert f\Vert_{L^q(\Omega)}\right)|x-y|^{\gamma}
 \end{align*} as desired.
\end{proof}

\begin{rem}
(1) Without the assumption on the sign of  $\operatorname{div} \mathbf{F}$,   there are examples which fail to satisfy the growth condition \eqref{eq: mu growth} (see \cite[Proposition~5.1]{PT}). Consider 
$ \mathbf{F}(x):=\frac{x}{|x|}\cos\!\left(\frac{1}{|x|^{\varepsilon}}\right)$, 
where $0<\varepsilon<n-1$ is fixed. A direct computation gives  
\[
    \operatorname{div} \mathbf{F}(x)
    =\varepsilon |x|^{-1-\varepsilon}\sin(|x|^{-\varepsilon})
    +(n-1)|x|^{-1}\cos(|x|^{-\varepsilon}).
\] 
Let  
\[
    r_k=\left(\frac{\pi}{6}+2k\pi\right)^{-1/\varepsilon}, \qquad k=1,2,3,\dots.
\] 
One can check that  
\[
    \int_{B_{r_k}}(\operatorname{div} \mathbf{F})^+\,\mathrm{d}x
    \ge \frac{\omega_n \varepsilon}{14(n-1-\varepsilon)}r_k^{\,n-1-\varepsilon},
\] 
where $\omega_n$ denotes the surface area of the unit sphere in $\mathbb{R}^n$.  
For $n\ge 3$, take $\varepsilon=1$. Then there exists a sequence $r_k\to 0$ such that  
\[
    \int_{B_{r_k}}|\operatorname{div} \mathbf{F}|\,\mathrm{d}x
    \ge \int_{B_{r_k}}(\operatorname{div} \mathbf{F})^+\,\mathrm{d}x
    \ge C r_k^{\,n-2}.
\]
Hence, the condition \eqref{eq: mu growth} does not hold for general 
$\mathbf{F}\in L^{\infty}(\Omega;\mathbb{R}^n)$ in higher dimensions.

(2) It is easy to see that the non-negativity of $\operatorname{div} \mathbf{F}$ can be replaced by $\operatorname{div} \mathbf{F}$ is bounded from below.
\end{rem}

\vskip 7pt

\section{Application to Singular Abreu equations}
In this section, we use Theorem~\ref{thm:divF+f} to derive the interior estimates for the following singular Abreu equations:
\begin{equation}\label{Abreu-plaplace}
\left\{
 \begin{alignedat}{2}
U^{ij} D_{ij} w&=- \div (DF(Du)) + Q(x, u, Du)=:f(x, u, Du, D^2u) ~&&\text{\ in} ~\ \ \Omega, \\\
 w~&= (\det D^2 u)^{-1}~&&\text{\ in}~\ \ \Omega,
 \end{alignedat}
\right.
\end{equation}
where $U=(U^{ij})= (\det D^2 u) (D^2 u)^{-1}$, $F\in W^{2, n}_{\text{loc}}(\mathbb R^n)$ is a convex function, and $Q$ is a function defined on $\mathbb R^n\times \mathbb R\times \mathbb R^n$. 
When the right-hand side $f$ depends only on the independent variable, that is $f=f(x)$, \eqref{Abreu-plaplace} is the {\it Abreu equation} arising from the problem of finding extremal metrics on toric manifolds in K\"ahler geometry \cite{Ab}, and it is equivalent to
$$\sum_{i,j=1}^n\frac{\partial^2 u^{ij}}{\partial x_i\partial x_j}=f(x),$$
where $(u^{ij})$ is the inverse matrix of $D^2u$. 
The general form in \eqref{Abreu-plaplace} was introduced by Le in \cite{Le2} in the study of convex functionals with a convexity constraint related to the Rochet-Chon\'e model \cite{RC} for the monopolist's problem in economics, whose Lagrangian depends on the gradient variable.

More specifically, in the calculus of variations with a convexity constraint, one considers minimizers of convex functionals
\[ \int_{\Omega} F_0(x, u(x), Du(x)) \,\mathrm{d}x\]
among certain classes of convex competitors, where $F_0(x,z,\mathbf{p})$ is a function on $\overline{\Omega}\times \mathbb R\times \mathbb R^n$. 
One example is the Rochet-Chon\'e model with $q$-power ($q>1$) cost  
 \[F_{q,\gamma}(x,z,\mathbf{p})=\left(\frac{|\mathbf{p}|^q}{q}-x\cdot \mathbf{p}+z\right)\gamma(x),\] 
 where $\gamma$ is a nonnegative and Lipschitz function called the relative frequency of agents in the population.
 
Since it is in general difficult to handle the convexity constraint, especially in numerical computations, instead of analyzing these functionals directly, one might consider analyzing their perturbed versions by adding the penalization $-\varepsilon\int_\Omega \ln \det D^2u \,\mathrm{d}x$ which are convex functionals in the class of $C^2$, strictly convex functions. The heuristic idea is that the logarithm of the Hessian determinant should act as a good barrier for the convexity constraint.
Note that critical points, with respect to compactly supported variations, of the convex functional
\[ \int_{\Omega} F_0(x, u(x), Du(x)) \,\mathrm{d}x -\varepsilon\int_\Omega \ln \det D^2u \,\mathrm{d}x,\]
satisfy the Abreu-type equation
\[\varepsilon U^{ij} D_{ij} [(\det D^2 u)^{-1}]= -\sum_{i=1}^n \frac{\partial}{\partial x_i} \big(\frac{\partial F_0}{\partial p_i} (x, u, Du)\big) +  \frac{\partial F_0}{\partial z}(x, u, Du).\]
Here we denote $\mathbf{p} =(p_1,\ldots, p_n)\in\mathbb{R}^n$. 
In particular, for the Rochet-Chon\'e model with $q$-power ($q>1$) cost and unit frequency $\gamma\equiv 1$, that is, $F_0=F_{q,1}$, the above right-hand side is 
\[-\operatorname{div} (|Du|^{q-2} Du)+ n+1,\] which belongs to the class of right-hand sides considered in \eqref{Abreu-plaplace}. When $F_0(x, z, \mathbf{p})=F(\mathbf{p}) + \hat F(x, z)$ 
the above right-hand side becomes \[-\operatorname{div} (DF(Du)) + \frac{\partial \hat F}{\partial z}(x, u).\]

The Abreu type equations can be included in a class of fourth-order Monge-Amp\`ere type equations of the form
\begin{equation}\label{4-eq-g}
U^{ij}D_{ij}[g(\det D^2 u)]=f
\end{equation} 
where $g:(0,\infty)\rightarrow (0,\infty)$ is an invertible function. In particular, when $g(t)=t^{\theta}$, one can take
  $\theta=-1$ and  $\theta=-\frac{n+1}{n+2}$ to get the Abreu type equation and the {\it affine mean curvature} type equation \cite{Ch}, respectively.
  It is convenient to write (\ref{4-eq-g}) as a system of two equations for $u$ and $w=g(\det D^2 u)$. One is a Monge-Amp\`ere equation for the convex function $u$ in the form of 
 \begin{equation*}
 \det D^2 u=
g^{-1}(w)
\end{equation*}
 and other is the following linearized Monge-Amp\`ere equation for $w$:
\begin{equation*}
U^{ij} D_{ij }w=f.
\end{equation*}
The second order linear operator  $U^{ij}D_{ij}$ is the linearized Monge-Amp\`ere operator associated with the convex function $u$ because its coefficient matrix is the cofactor matrix of  $D^2 u$.
The regularity and solvability of equation \eqref{4-eq-g}, under suitable boundary conditions, are closely related to the regularity theory of the linearized Monge-Amp\`ere equation, which is connected to the results obtained in the previous sections. Therefore, we present it as an example illustrating our application.

In the following, we assume that lower and upper bounds for the determinant of the Hessian have already been established-this being one of the main challenges in the regularity theory of singular Abreu equations. For simplicity, we also assume regular conditions on the functions $F$ and $Q$. We remark that the result below is not new; for example, it can also be obtained using the transformation method developed in \cite{KLWZ}. Nonetheless, we still present it here to illustrate a straightforward application of Theorem~\ref{thm:divF+f}.

\begin{thm}
\label{inter-abreu}
Assume that $\Omega\subset\mathbb R^n$ is a uniformly convex domain with boundary $\partial\Omega\in C^3$. Let $F\in  W^{2, r}_{\text{loc}}(\mathbb R^n)$ be a concex function for some $r>n$, and let $Q\in L^s_{\text{loc}}(\mathbb R^n\times \mathbb R\times \mathbb R^n)$ where $s>n$. 
Assume that $u\in W^{4, s}(\Omega)$ is a uniformly convex solution to the singular Abreu equation \eqref{Abreu-plaplace}. Suppose that, for some positive constants $\lambda$, $\Lambda$, we have
\[\lambda \leq \det D^2 u  \leq \Lambda \quad \text{in}\quad \Omega.\]
\begin{enumerate}
\item[(i)]  Assume $F\in C^{2,\alpha_0}(\mathbb R^n)$ and $Q\in C^\alpha(\mathbb R^n\times \mathbb R\times \mathbb R^n)$ where $\alpha_0, \alpha\in (0, 1)$.
Then for any $\Omega'\Subset\Omega$, there exist 
constants $\beta, C >0 $ depending only on  $\alpha$, $\alpha_0$, $\lambda$, $\Lambda$, $n$,  $r$, $\|u\|_{L^\infty(\Omega)}$, $F$, $Q$, $\operatorname{dist}(\Omega',\partial\Omega)$ and the modulus of convexity of $u$ such that 
\[\|u\|_{C^{4,\beta}(\Omega')}\leq C. \]

\item[(ii)]
For any $\Omega'\Subset\Omega$, there exist 
constants $p, C >0$ depending only on  $\lambda$, $\Lambda$, $n$,  $r$, $s$, $\|u\|_{L^\infty(\Omega)}$, $F$, $Q$, $\operatorname{dist}(\Omega',\partial\Omega)$ and the modulus of convexity of $u$ such that 
\[\|u\|_{W^{4, p}(\Omega')}\leq C. \]
\end{enumerate}

\end{thm}

\begin{proof}[Sketch of the proof]
    By \cite{Le2, Le23, LZ, KLWZ}, it suffices to prove that $w$ is locally Hölder continuous. Once this is established, we can apply Caffarelli’s Schauder estimates for Monge-Ampère equations to deduce that $D^2 u$ is locally Hölder continuous. Consequently, the first equation in \eqref{Abreu-plaplace} becomes a uniformly elliptic equation with Hölder continuous coefficients, from which higher-order derivative estimates follow. 

    To establish the Hölder estimate for $w$, we observe that since $u$ is convex, it is locally Lipschitz. From the convexity of $F$ and $u$, we have \[\operatorname{div}(DF(Du))=\operatorname{tr}(D^2F(Du)D^2u)\ge 0.\] These imply that both $DF(Du(x))$ and $Q(x, u(x), Du(x))$ satisfy the assumptions of Theorem \ref{thm:divF+f}. Therefore, we can directly apply Theorem \ref{thm:divF+f} to conclude that $w \in C^{\alpha}(\Omega')$.
\end{proof}

\vspace{1em}
\noindent\textbf{Acknowledgments.} This research is partially supported by  National Key R$\&$D Program of China 2020YFA0712800, 2023YFA009900 and NSFC  Grant 12271008. Also, Ling Wang was funded by the European Union: the European Research Council (ERC), through StG “ANGEVA”, project number: 101076411. Views and opinions expressed are however those of the authors only and do not necessarily reflect those of the European Union or the European Research Council. Neither the European Union nor the granting authority can be held responsible for them. The authors would like to thank Professor Nam Le for many helpful comments.

\end{document}